\documentclass[11pt]{amsart}
\usepackage{amssymb}
\usepackage{verbatim}

\newtheorem{lemma}{Lemma} [section]
\newtheorem{thm}[lemma]{Theorem}
\newtheorem{conjecture}[lemma]{Conjecture}

\newtheorem{defi}[lemma]{Definition}

\theoremstyle{remark}

\newtheorem*{remark}{Remark}

\numberwithin{equation}{section}

\DeclareMathOperator{\pr}{{pr}}

\DeclareMathOperator{\id}{{id}}

\newcommand{\N}{{\mathbb N}}
\newcommand{\Q}{{\mathbb Q}}
\newcommand{\bG}{{\mathbb G}}
\newcommand{\Qbar}{{\overline\Q}}

\newcommand{\C}{{\mathbb C}}

\newcommand{\M}{{\mathbb M}}

\newcommand{\lra}{\longrightarrow}
\newcommand{\A}{{\mathbb A}}
\newcommand{\bP}{{\mathbb P}}

\newcommand{\hh}{\widehat{h}}

\DeclareMathOperator{\nt}{nt}

\newcommand{\bC}{{\mathbb C}}

\newcommand{\bA}{{\mathbb A}}
\newcommand{\bQ}{{\mathbb Q}}

\newcommand{\cO}{\mathcal{O}}

\newcommand{\cL}{\mathcal{L}}

\newcommand{\cJ}{\mathcal{J}}

\newcommand{\bQb}{\overline{\bQ}}

\begin{document}

\title[Proof of a dynamical Bogomolov conjecture]{Proof of a
    dynamical Bogomolov conjecture for lines under polynomial actions}

\author{Dragos Ghioca}

\address{
Dragos Ghioca \\
Department of Mathematics \& Computer Science\\
University of Lethbridge \\
Lethbridge, AB T1K 3M4 
}

\email{dragos.ghioca@uleth.ca}

\author{Thomas J. Tucker}

\address{
Thomas Tucker\\
Department of Mathematics\\
Hylan Building\\
University of Rochester\\
Rochester, NY 14627
}

\email{ttucker@math.rochester.edu}

\thanks {The first author was partially supported by NSERC.
The second author was partially supported by NSA
    Grant 06G-067 and NSF Grant DMS-0801072.}
\begin{abstract}
  We prove a dynamical version of the Bogomolov conjecture
  in the special case of lines in $\bA^m$ under the action of a map
  $(f_1,\dots,f_m)$ where each $f_i$ is a polynomial in $\Qbar[X]$ of the same degree.
\end{abstract}

\date{\today} 
\maketitle

\section{Introduction}

In 1998, Ullmo \cite{Ullmo} and Zhang \cite{Zhang} proved the
following conjecture of Bogomolov \cite{Bog}.
\begin{thm} Let $A$ be an abelian variety defined over a
  number field with N\'eron-Tate height $\hh_{\nt}$ and let $W$ be a
  subvariety of $A$ that is not a torsion translate of an abelian
  subvariety of $A$.  Then there exists an $\epsilon > 0$
  such that the set
$$ \{ x \in A(\bQb) \; \mid \; \hh_{\nt}(x) \le \epsilon \}$$
is not Zariski dense in $W$.  
\end{thm}
Earlier, Zhang \cite{zhangvar} had proved a similar result for the
multiplicative group $\bG_m^n$.  Zhang \cite{zhangadelic, ZhangLec}
also proposed a more general conjecture for what he called {\it
  polarizable} morphisms; a morphism $\Phi: X \lra X$ on a projective
variety $X$ is said to be polarizable if there is an ample line bundle
$\cL$ on $X$ such that $\Phi^* \cL \cong q \cL$ for some integer $q > 1$.
When a polarizable map $\Phi$ is defined over a number field, it gives
rise to a canonical height $\hh_\Phi$ with the property that
$\hh_\Phi(\Phi(\alpha)) = q \hh_\Phi(\alpha)$ for all $\alpha \in
X(\bQb)$.  Zhang makes the following Bogomolov-type conjecture in this
more general context.

\begin{conjecture}\label{conj}(Zhang) Let $\Phi: X \lra X$ be a polarizable
  morphism of a projective variety defined over a number field and let
  $W$ be a subvariety of $X$ that is not preperiodic under $\Phi$.
  Then there exists an $\epsilon > 0$ such that the set
$$ \{ x \in W(\bQb) \; \mid \; h_{\Phi}(x) \le \epsilon \}$$
is not Zariski dense in $W$.
\end{conjecture}

The definition of preperiodicity for varieties here is the same as the
usual definition of preperiodicity for points.  More precisely, for
any quasiprojective variety $X$, any endomorphism $\Phi:X\lra X$, and
any subvariety $V\subset X$, we say that $V$ is
\emph{$\Phi$-preperiodic} if there exists $N\ge 0$, and $k\ge 1$ such
that $\Phi^N(V)=\Phi^{N+k}(V)$.  Note that when $A$ is an abelian
variety and $\Phi$ is a multiplication-by-$n$ map (for $n\ge 2$), a subvariety $W$ is
preperiodic if and only if it is a torsion translate of an abelian
subvariety of $A$.

In this paper, we prove the following special case of
Conjecture~\ref{conj}.   
\begin{thm}
\label{MM}
Let $f_1,\dots,f_m\in \Qbar[X]$ be polynomials of degree $d > 1$,
let $\Phi:=(f_1,\dots,f_m)$ be their coordinatewise action on
$\A^m$, and let $L$ be a line in $\A^m$ defined over $\Qbar$.  If $L$
is not $\Phi$-preperiodic, then there exists an $\epsilon > 0$ such
that 
$$ S_{L,\Phi,\epsilon}:=\{ x \in L(\bQb) \; \mid \; \hh_\Phi(x) \le \epsilon \}$$
is finite (see Section~\ref{pre} for the definition of $\hh_{\Phi}$).
\end{thm}

Baker and Hsia \cite[Theorem 8.10]{Baker-Hsia} previously proved
Theorem~\ref{MM} in the special case where $f_1 = f_2$ and $m=2$.
\\
\\
\noindent\textsl{Acknowledgements.}  We would like to thank Arman
Mimar for helpful conversations.

\section{Preliminaries}
\label{pre}
\noindent {\bf Heights.}  Let $\M_\bQ$ be the usual set of absolute
values on $\bQ$, normalized so that the archimedean absolute value is
simply the absolute value $| \cdot|$ and $| p |_p = 1/p$ for each
$p$-adic absolute value $| \cdot |_p$.  For any extension $K$ of $\bQ$
we define $\M_K$ to be the set of absolute values on $K$ that extend
elements of $\M_{\bQ}$.  Then, for any $x \in\bQb$ we define the Weil
height of $x$ to be
$$h(x) = \frac{1}{[\bQ(x):\bQ]} \cdot \sum_{v\in
  \M_{\bQ(x)}}\sum_{\substack{w|v\\ w\in \M_{\bQ(x)}}}
\log\max\{|x|_w^{[\bQ(x)_w:\bQ_v]} , 1\}$$
where $\bQ_v$ and $\bQ(x)_w$ are the completions of $\bQ$ and $\bQ(x)$ at $v$
and $w$ respectively (see \cite[Chapter 1]{BG} for details).  

For a polynomial $f \in\bQb[X]$ of degree greater than 1, define
the $f$-canonical height $\hh_f:\Qbar\lra\mathbb{R}_{\ge 0}$ by
\begin{equation}\label{def}
\hh_f(x) = \lim_{n \to \infty} \frac{h(f^n(x))}{(\deg f)^n},
\end{equation}
following Call-Silverman \cite{CS} (where $f^n$ denotes the $n$-th iterate of $f$).  
 
Let $f_1, \dots, f_m \in \bQb[X]$ be polynomials of degree $d > 1$,  and
let $\Phi:=(f_1,\dots,f_m)$ be their coordinatewise action on
$\A^m$; that is, 
$$ \Phi(x_1, \dots, x_m) = (f_1(x_1), \dots, f_m(x_m)).$$
We define the $\Phi$-canonical height
$\hh_{\Phi}:\A^m(\Qbar)\lra\mathbb{R}_{\ge 0}$ by
$$\hh_{\Phi}(x_1,\dots,x_m)= \sum_{i=1}^m\hh_{f_i}(x_i) .$$

Note that while $\bA^m$ is not a projective variety, $\Phi$ extends to
a map ${\tilde \Phi}: (\bP^1)^m \lra (\bP^1)^m$.  Furthermore,
${\tilde \Phi}$ is polarizable, since 
$${\tilde \Phi}^* \bigotimes_{i=1}^m \pr_i^* \cO_{\bP^1}(1) 
\cong \bigotimes_{i=1}^m \pr_i^* \cO_{\bP^1}(d),$$ 
where $\pr_i$ is the projection of $(\bP^1)^m$ onto its $i$-th
coordinate.

\begin{remark}
  Theorem~\ref{MM} is {\it not} true if one allows the polynomials
  $f_i$ to have different degrees.  This is easily seen, for example,
  in the case where $m=2$, the line $L$ is the diagonal, and $f_2 =
  f_1^2$.  The map $\Phi = (f_1, \dots, f_m)$ is only polarizable when
  $\deg f_i = \deg f_j$, so this is not a counterexample to
  Conjecture~\ref{conj}.
\end{remark}

\section{Symmetries of the Julia set}
\label{Julia}

In this section we recall the main results regarding polynomials which share the same Julia set. For a polynomial $f \in \bC[X]$, we let $J(f)$
denote the Julia set of $f$ (see \cite[Chapter 3]{Beardon} or
\cite{milnor} for the definition of a Julia set of a rational
function over the complex numbers).
As proved by Beardon \cite{Beardon-3, Beardon-2}, any family of polynomials which have the same Julia set $\cJ$ is determined by the symmetries of $\cJ$.
\begin{defi}
\label{symmetries of Julia sets}
If $f\in\C[z]$ is a polynomial and $J(f)$ is its Julia set, then the symmetry group $\Sigma(f)$ of $J(f)$ is defined by
$$\Sigma(f)=\{\sigma\in \mathfrak{C}\text{ : }\sigma(J(f))=J(f)\},$$
where $\mathfrak{C}$ is the group of conformal Euclidean isometries.
\end{defi}

Beardon \cite[Lemma 3]{Beardon-3} computed $\Sigma(f)$ for any $f\in \C[z]$. 
\begin{lemma}
\label{l-3 b-3}
The isometry group $\Sigma(f)$ is a group of rotations about some point $\zeta\in\C$, and it is either trivial, or finite cyclic group, or the group of all rotations about $\zeta$.
\end{lemma}

To prove the result above, Beardon uses the fact that each polynomial
$f$ of degree $d\ge 2$ is conjugate to a monic polynomial
$\tilde{f}\in\C[z]$ which has no term in $z^{d-1}$. Clearly, if
$\tilde{f}=\gamma \circ f \circ \gamma^{-1}$ where $\gamma\in\C[z]$ is
a linear polynomial, then $J(\tilde{f})=\gamma(J(f))$; thus it
suffices to prove Lemma~\ref{l-3 b-3} for $\tilde{f}$ instead of $f$.
If $\tilde{f}(z)=z^d$, then $J(\tilde{f})$ is the unit circle, and
$\Sigma(\tilde{f})$ is the group of all rotations about $0$ (see also
\cite[Lemma 4]{Beardon-3}).  If $\tilde{f}$ is not a monomial, then we
choose $b\ge 1$ maximal such that we can find $a\ge 0$ and
$\tilde{f}_1\in\C[z]$ satisfying
$\tilde{f}(z)=z^a\tilde{f}_1(z^b)$. In this case, $J(\tilde{f})$ is
the finite cyclic group of rotations generated by the multiplication
by $\exp(2\pi i/b)$ on $\C$ (see \cite[Theorem 5]{Beardon-3}).
Beardon also proves the following result (see \cite[Lemma
7]{Beardon-3}) which we will use later.
\begin{lemma}
\label{l-7 b-3}
If $f\in\C[z]$ is a polynomial of degree $d\ge 2$ and $\sigma\in J(f)$, then $f\circ \sigma = \sigma^d \circ f$. 
\end{lemma}

The following classification of polynomials which have the same Julia set is proven by Beardon in \cite[Theorem 1]{Beardon-2}.
\begin{lemma}
\label{t-1 b-2}
If $f,g\in \C[z]$ have the same Julia set, then there exists $\sigma\in\Sigma(f)$ such that $g=\sigma \circ f$.
\end{lemma}

\section{Proof of our main result} 
\label{proof main}

\begin{proof}[Proof of Theorem~\ref{MM}.]
  Suppose that for every $\epsilon>0$, the set $S_{L,\Phi,\epsilon}$
  is infinite.  We will show that this implies that $L$ must be
  $\Phi$-preperiodic.

We first note that it suffices to prove the theorem for the line
$L' = (\sigma_1, \dots, \sigma_m)(L)$ and the map 
$$\Phi' = (\sigma_1 \circ f_1 \circ \sigma_1^{-1}, \dots, \sigma_m \circ f_m \circ \sigma_m^{-1})$$ 
for some linear automorphisms $\sigma_1, \dots,
\sigma_m$ of $\bA^1$.  This follows from the fact that $L$ is
preperiodic for $\Phi$ if and only if $ L'$ is
preperiodic for $\Phi'$ along with the equality
\begin{equation}\label{change}
 \hh_{\Phi'}(\sigma_1 \alpha_1, \dots, \sigma_m \alpha_m) =  \hh_{\Phi}(\alpha_1,
\dots, \alpha_m).
\end{equation}
Note that \eqref{change} is a simple consequence of Definition~\ref{def},
since $|h(\sigma_i x) - h(x)|$ is bounded for all $x \in \bQb$.

We now proceed by induction on $m$; the case $m=1$ is obvious.

If the projection of $L$ on any of the coordinates consists of only
one point, we are done by the inductive hypothesis. Indeed, without
loss of generality, assume the projection of $L$ on the first
coordinate equals $\{z_1\}$, then $L=\{z_1\}\times L_1$, where
$L_1\subset \A^{m-1}$ is a line, and $\hh_{f_1}(z_1)=0$. Since only
preperiodic points have canonical height equal to $0$ (see \cite[Cor.\ $1.1.1$]{CS}), we conclude
that $z_1$ is $f_1$-preperiodic, and thus we are done by the induction
hypothesis applied to $L_1$.

Suppose now that $L$ projects dominantly onto each coordinate of
$\A^m$. For each $i=2,\dots,m$, we let $L_i$ be the projection of $L$
on the first and the $i$-th coordinates of $\A^m$. Then $L_i$ is a
line given by an equation $X_1=\sigma_i(X_i)$, for some linear
polynomial $\sigma_i\in\Qbar[X]$. Clearly, it suffices to show that
for each $i=2,\dots,m$, the line $L_i$ is preperiodic under the action
of $(f_1,f_i)$ on the corresponding two coordinates of $\A^m$.

Let $\tilde{f_i}:=\sigma_i \circ f_i \circ \sigma_i^{-1}$ and let $\Delta = (x,x)
\in \bA^2$ be the diagonal on $\bA^2$.  By our remarks at the
beginning of the proof, it suffices to show that $(\id, \sigma_i) (L_i)
= \Delta$ is preperiodic under the action of $(f_1,\tilde{f_i})$.
Furthermore, the fact that we have an infinite sequence
$(z_{n,1},z_{n,i})\in L_i(\Qbar)$ with
$$\lim_{n\to\infty} \hh_{f_1}(z_{n,1}) = \lim_{n\to\infty}\hh_{f_i}(z_{n,i})=0$$
means that we have
$$\lim_{n\to\infty}\hh_{\tilde{f_i}}(z_{n,1})=0,$$
because of \eqref{change}.  Fix an embedding $\theta: \bQb \lra \bC$ and let $f_1^\theta$ and ${\tilde f_i}^\theta$ be the images of
$f_1$ and ${\tilde f_i}$, respectively, in $\bC[X]$ under this embedding.  Then, by
\cite[Corollary 4.6]{Baker-Hsia}, the Galois orbits of the points
$\{z_{n,1}\}_{n\in\N}$ are equidistributed with respect to the
equilibrium measures on the Julia sets of both $f_1^\theta$ and
$\tilde{f_i}^\theta$.  Since the support of the equilibrium measure $\mu_g$
of a polynomial $g \in \bC[X]$ is equal to the Julia set of $g$ (\cite[Section
4]{Baker-Hsia}), we must have $J({\tilde f_i}^\theta) = J(f_1^\theta)$.

By Lemma~\ref{t-1 b-2} (see also \cite{Baker-Eremenko, pau}),
there exists a conformal Euclidean symmetry $\mu_i: z \lra a_iz + b_i$
such that $\mu_i(J(f_1^\theta)) = J(f_1^\theta)$ and ${\tilde
  f_i^\theta}= \mu_i \circ f_1^\theta$.  Note that $a_i$ and $b_i$ must be in
the image of $\bQb$ under $\theta$ since the coefficients of
$f_1^\theta$ and $f_i^\theta$ are.  Let $\tau_i$ be the map $\tau_i: z
\lra \theta^{-1}(a_i) z + \theta^{-1}(b_i)$.  Then we have $\tilde{f_i} =
\tau_i \circ f_1$.

 If $\tau_i$ has infinite order, then it follows from \cite[Lemma
 4]{Beardon-3} that there exist linear polynomials $\gamma_1,
 \gamma_i$ such that $\gamma_1 \circ f_1 \circ \gamma_1^{-1} = \gamma_i \circ
 \tilde{f_i} \circ \gamma_i^{-1} = X^d$. In this case, we reduce our problem
 to the usual Bogomolov conjecture for $\bG_m^2$, proved by Zhang
 \cite{zhangthese}.  Indeed, Zhang proves that if a curve $C$ in
 $\bG_m^2$ has an infinite family of algebraic points with height
 tending to zero, then it must be a torsion translate of an algebraic
 subgroup of $\bG_m^2$; that is, $C = \xi A$ where $\xi$ has finite
 order and $A$ is an algebraic subgroup of $\bG_m^2$.  Since $(\xi
 A)^n = \xi^n A$ and $\xi$ has finite order, it is clear that such a
 curve is preperiodic under the map $(X,Y) \mapsto (X^d, Y^d)$.

 We may suppose then that $\tau_i$ has finite order. By Lemma~\ref{l-7 b-3}, we have $f_1 \circ \tau_i = \tau_i^{d} \circ
 f_1$. Thus, we have
$$\tilde{f_i}^k = \tau_i^{(d^k - 1)/(d-1)} \circ f_1^k$$
for all $k \ge 1$.  Since $\tau_i$ has finite order, we conclude that
the set 
$$\{ \tau_i^{(d^k-1)/(d-1)}\}_{k\ge 0}$$
is finite. This implies that the set of curves of the form $ (f_1^k,
\tilde{f_i}^k)(\Delta)$ is finite, which means the diagonal subvariety
$\Delta$ is preperiodic under the action of $(f_1,\tilde{f_i})$, as desired.
\end{proof}

We believe that it is possible to extend the methods of the proof of
Theorem~\ref{MM} to the case of arbitrary rational maps $\varphi_1,
\dots, \varphi_m$ of the same degree, though the proof seems to be
much more difficult, requiring in particular Mimar's \cite{Mimar}
results on arithmetic intersections of metrized line bundles and an
analysis of Douady-Hubbard-Thurston's \cite{DH} classification of
critically finite rational maps with parabolic orbifolds.  We intend
to treat this problem in a future paper.


\def\cprime{$'$} \def\cprime{$'$} \def\cprime{$'$} \def\cprime{$'$}
\providecommand{\bysame}{\leavevmode\hbox to3em{\hrulefill}\thinspace}
\providecommand{\MR}{\relax\ifhmode\unskip\space\fi MR }
\providecommand{\MRhref}[2]{%
  \href{http://www.ams.org/mathscinet-getitem?mr=#1}{#2}
}
\providecommand{\href}[2]{#2}

\end{document}